\documentclass[a4paper,10pt,reqno]{article}

\usepackage{amsmath,amsthm,amsfonts,amssymb,amsopn,color}
\usepackage[pagewise]{lineno}

\usepackage{amssymb,amsfonts}
\usepackage[all,arc]{xy}
\usepackage{esint,enumerate}
\usepackage{mathrsfs}
\usepackage{amsthm}
\usepackage{lipsum}
\usepackage{pgf,tikz}
\usetikzlibrary{arrows}
\usepackage{graphicx,mfpic}
\usepackage{subfigure}
\usepackage[utf8]{inputenc}

\makeatletter
\g@addto@macro\normalsize{%
\setlength\abovedisplayskip{4pt}
\setlength\belowdisplayskip{4pt}
\setlength\abovedisplayshortskip{4pt}
\setlength\belowdisplayshortskip{4pt}
}
\makeatother

\newtheorem{theorem}{Theorem}[section]
\newtheorem{lemma}[theorem]{Lemma}

\newtheorem{proposition}[theorem]{Proposition}

\newtheorem*{thma}{Theorem A}

\theoremstyle{definition}

\newtheorem{remark}[theorem]{Remark}

\newcommand{\rone}{\mathbb{R}}
\newcommand{\rtwo}{{\mathbb R^2}}

\newcommand{\cpx}{\mathbb C}
\newcommand{\nat}{\mathbb N}

\newcommand{\calX}{{\mathcal X}}

\newcommand{\tila}{{\widetilde a}}

\newcommand{\tilu}{{\widetilde u}}

\newcommand{\re}{{\mathrm{Re }}}

\newcommand{\al}{{\alpha}}

\newcommand{\la}{{\lambda}}
\newcommand{\ve}{{\varepsilon}}
\newcommand{\vp}{{\varphi}}
\newcommand{\om}{{\omega}}

\newcommand{\Om}{{\Omega}}

\newcommand{\pa}{\partial}

\newcommand{\qand}{{\quad \mbox{and} \quad}}
\newcommand{\qas}{{\quad \mbox{as} \quad}}
\newcommand{\qfor}{{\quad \mbox{for} \quad}}

\newcommand{\qin}{{\quad \mbox{in} \quad}}
\newcommand{\qon}{{\quad \mbox{on} \quad}}

\begin{document}

\title{\textbf{Quantization effects for multi-component Ginzburg-Landau vortices}}

\author{Rejeb Hadiji  \thanks{Univ Paris Est Creteil, CNRS, LAMA, F-94010 Creteil, France. 
Univ Gustave Eiffel, LAMA, F-77447 Marne-la-Vall'ee, France.
 email: hadiji@u-pec.fr}, $\,$
 Jongmin Han\thanks{Department of Mathematics, Kyung Hee University,
 Seoul, 02447, Korea.
 email:  jmhan@khu.ac.kr} $\,$ and $\,$
Juhee Sohn\thanks{School of Liberal Arts and Sciences, Korea National University of Transportation,
Chungju, 27469, Korea. email: woju48@ut.ac.kr}
}
 \date{}
\maketitle

\begin{abstract}
In this paper, we are concerned with $n$-component Ginzburg-Landau equations on $\rtwo$.
By introducing a diffusion constant  for each component, we discuss that the $n$-component equations are different from $n$-copies of the single Ginzburg-Landau equations.
Then, the results of Brezis-Merle-Riviere for the single Ginzburg-Landau equation can be nontrivially extended to the multi-component case.
First, we show that if the solutions have their gradients in $L^2$ space, they are trivial solutions.
Second, we prove that if the potential is square summable, then it has quantized integrals, i.e., there exists one-to-one correspondence between the possible values of the potential energy and $\nat^n$.
Third, we show that different diffusion coefficients in the system are important to obtain nontrivial solutions of $n$-component equations.
\end{abstract}

\small



\normalsize

\section{Introduction}
\setcounter{equation}{0}

Let $\Om \subset \rtwo$ be a smooth bounded simply connected domain.
The classical Ginzburg-Landau(abbreviated by GL) energy on $\Om$  is given by
\begin{equation}
\label{eq:ftnal E BBH}
 G_\ve^b (u)  = \frac12 \int_\Om  |\nabla u|^2  dx + \frac{1}{4\ve^2} \int_\Om \big(1- |u|^2 \big)^2 dx.
\end{equation}
Here, $u :\Om \to \rtwo$ is an order parameter and $u=g$ on $\pa \Om$ for a smooth function $g:\Om \to S^1$.
For convenience, we often regard  $u:\Om \to \rtwo$   as   $u:\Om \to \cpx$.
The corresponding Euler-Lagrange equations are
\begin{equation}
\label{eq:GL-BBH}
\left\{
\begin{aligned}
-\Delta u & = \frac{1}{\ve^2} u  (1-|u |^2  ) \qin  \Om,\\
u&=g \qon \pa \Om.
\end{aligned} \right.
\end{equation}
For the last three decades, there have been lots of studies on the solution structures of \eqref{eq:GL-BBH}, in particular, on the asymptotics as $\ve \to 0$.

Such studies started from the seminal work of \cite{BBH93} and \cite{BBH94} which deals fairly in detail with the asymptotic behaviors of   minimizers $u_\ve$ for $G_\ve^b$ on
\[ H_g^1(\Om, \rtwo)= \{ u \in H^1(\Om, \rtwo): u=g \text{ on } \pa\Om\}.
\]
The main issue in the analysis is the value of the degree of $g$, say $d$.
If $d=0$, then $H_g^1(\Om, S^1)\neq \emptyset$ and $u_\ve$ converges to a  harmonic map $u_0$ in $H_g^1(\Om, \rtwo)$.
Here, $u_0$  minimizes
\[ \int_\Om |\nabla u|^2 dx \qon H_g^1(\Om,S^1)
\]
and is the unique solution of
\begin{equation}
\label{eq:u0 GL}
\begin{aligned}
-\Delta u_0  = u_0 |\nabla u_0|^2   \text{ on }  \Om, \quad
u_0  =g    \text{ on }  \pa\Omega, \quad
|u_0| =1  \text{ on } \Om.
\end{aligned}
\end{equation}
The main point is that $u_0$ serves as a test function for $G_\ve^b$.
This gives a uniform  boundedness of $G_\ve^b(u_\ve)$ and hence $H^1$-boundedness of $u_\ve$.
However, if $d\neq0$, then $H_g^1(\Om, S^1)=\emptyset$ and the energy is not bounded:
\begin{equation}
\label{eq:G-BBH d>0}
G_\ve^b(u_\ve)=2\pi d\ln\frac{1}{\ve}+O(1).
\end{equation}
This phenomena gives rise to the formation of singularities $\{a_1, \cdots, a_n\} \subset \Om$ such that
 $u_\ve \to u_*\in H_{loc}^1(\Om\setminus \{a_1, \cdots, a_n\})$ where $u_*$ is a harmonic map that has singularities $a_1, \cdots, a_n$.

 On the other hand, if $\Om$ is star-shaped, then the Pohozaev identity implies  that the potential is uniformly bounded in the limit $\ve \to 0$:
 \begin{equation}
\label{eq:GL-potential}
E_\ve^b(u_\ve) = \frac{1}{ \ve^2} \int_\Om \big(1- |u_\ve|^2 \big)^2 dx \le C.
\end{equation}
This also implies by \eqref{eq:G-BBH d>0} that
 \begin{equation}
\label{eq:GL-grad}
  \int_\Om |\nabla u_\ve|^2 \to \infty \qas \ve\to0.
\end{equation}
If we set $\tilu_\ve(x)=u_\ve(\ve x)$, then $\tilu_\ve$ satisfies
\begin{equation}
\label{eq:scale GL}
-\Delta \tilu_\ve  =  \tilu_\ve (1-|\tilu_\ve|^2  ) \qin \Omega_\ve=\frac{1}{\ve}\Om.
\end{equation}
If we take the limit $\ve \to 0$, $\tilu_\ve$ converges to a function $u$ in $C^2_{loc}(\rtwo)$ such that
\begin{equation}\label{eq:BBH-rtwo-before-lambda}
-\Delta u =  u (1-|u|^2) \qin \rtwo.
\end{equation}
More generally, let us consider for $\la>0$
\begin{equation}\label{eq:BBH-rtwo}
-\la \Delta u = u (1-|u|^2) \qin \rtwo.
\end{equation}
This can  be derived from the limit $\ve \to 0$ for a variant of \eqref{eq:GL-BBH}:
\begin{equation}
\label{eq:GL weight}
\left\{
\begin{aligned}
- \mbox{div}  \big( a (x)  \nabla u \big) & = \frac{1}{\ve^2} u   (1- |u |^2 ) \qin \Omega,\\
u&=g  \qon \pa\Omega.
\end{aligned}
\right.
\end{equation}
The weight function $a(x)>0$  appears in \eqref{eq:GL weight} when we study the equation  \eqref{eq:GL-BBH}  on a domain equipped with  a   background  Riemannian metric.
See \cite{AS98,AH95,AH00} for the study of asymptotics of solutions to \eqref{eq:GL weight}.
If we set   $\tila (x) = a (\ve x)$ and $\tilu_\ve(x)=u_\ve(\ve x)$ for a solution  $u_\ve $ of \eqref{eq:GL weight},  then $\tilu_\ve$ satisfies
\begin{equation}
\label{eq:scale GL}
- \mbox{div}  \big( \tila (x)  \nabla \tilu_\ve \big) =  \tilu_\ve  (1- |\tilu_\ve|^2 ) \qin \Omega_\ve.
\end{equation}
By taking the limit $\ve \to 0$, we are led to \eqref{eq:BBH-rtwo} with $\la =a (0) $.
Although \eqref{eq:BBH-rtwo} can be obtained directly from \eqref{eq:BBH-rtwo-before-lambda} by a scaling $u(x) \mapsto u(x/\sqrt{\la})$,  it is helpful to consider \eqref{eq:GL weight}  when we  study  $n$-component generalization  on $\rtwo$ as we shall see.

Meanwhile, since $E_\ve^b(u_\ve) $ is convergent up to a subsequence in view of \eqref{eq:GL-potential},  it is interesting to find the value of
 \begin{equation}
\label{eq:GL-potential-R2}
E^b(u ) =   \int_\rtwo\big(1- |u |^2 \big)^2 dx
\end{equation}
for solutions of \eqref{eq:BBH-rtwo}.
Remarkably, it was verified in \cite{BMR94} that if $u$ is a solution of \eqref{eq:BBH-rtwo}, then $E^b(u )$ allows only quantized values.
Meanwhile, in the point of \eqref{eq:GL-grad}, it was proved that if $u$ is a solution of \eqref{eq:BBH-rtwo} and $\nabla u \in L^2(\rtwo)$, then $u$ is a trivial solution.
This can be summarized  as
\begin{thma}{\rm \cite{BMR94}}
\begin{itemize}
\item[{\rm (i)}]
If $u$ is a solution of \eqref{eq:BBH-rtwo}, then $E^b(u )=2\pi \la d^2$ for $d=0,1,2, \cdots, \infty$.

\item[{\rm (ii)}]
If $u$ is a solution of \eqref{eq:BBH-rtwo} and $\nabla u \in L^2(\rtwo)$, then  either $u\equiv 0$ or $u\equiv c_0$ for a constant $c_0$ with $|c_0|=1$.
\end{itemize}
\end{thma}

The classical model with the energy \eqref{eq:ftnal E BBH} can be generalized as an $n$-component model.
For a pair of maps $(u_1,\cdots,u_n) \in H^1_{g_1}(\Om, \cpx) \times\cdots\times  H^1_{g_n}(\Om, \cpx)$, we consider an energy
\begin{equation}
\label{eq:ftnal G}
 G_\ve (u_1,\cdots,u_n) = \frac12 \int_\Om \sum_{j=1}^n |\nabla u_j|^2 dx + \frac{1}{4\ve^2} \int_\Om \Big(n-\sum_{j=1}^n |u_j|^2\Big)^2 dx.
\end{equation}
The Euler-Lagrange equations are
\begin{equation}
\label{eq:semilocal GL}
\left\{
\begin{aligned}
-\Delta u_i& = \frac{1}{\ve^2} u_i \Big(n-\sum_{j=1}^n |u_j|^2\Big) \qin \Omega,\\
u_i&=g_i  \qon \pa\Omega,
\end{aligned}
\right.
\end{equation}
for   $i=1,\cdots,n$.
This direct generalization to an $n$-component model originates from a semilocal gauge field model
\cite{Hi92, VaAch91}.
The energy \eqref{eq:ftnal G} was used to explain some issues in cosmology related to a  particular  state of the early universe that possesses both a local and a global gauge invariance.
It also describes interactions of multiple order parameters.
See \cite{HHS22,VaAch91} for more physical motivations.

It is very natural to ask that the properties of the classical model \eqref{eq:ftnal E BBH} are still valid for the generalized model \eqref{eq:ftnal G}.
Regarding this question, a detailed analysis was firstly performed in \cite{HHS22}.
A remarkable thing is that  the convergence of minimizers of \eqref{eq:ftnal G} is not so seriously affected by the degrees
\begin{equation}
\label{eq:deg gj}
\deg (g_j,\pa\Om)=d_j   \in \nat\cup\{0\}  \qfor  j=1,\cdots, n.
\end{equation}
Indeed, a natural generalization of  $H_g^1(\Om,S^1)$ is  the space
\begin{align*}
 &\calX(g_1,\cdots,g_n;\Om)\\
 =&  \Big\{  (u_1,\cdots,u_n) \in H^1_{g_1}(\Om; \rtwo) \times\cdots\times H^1_{g_n} (\Om; \rtwo) ~:~\sum_{j=1}^n |u_j |^2 = n~  \text{ a.e. on } \Om\Big\}.
\end{align*}
Then, it is known by \cite[Theorem 1.1]{HHS22} that $\calX(g_1,\cdots,g_n;\Om) \ne \emptyset$ regardless of the values of $d_j$'s.
For $n\ge 2$, the condition $\sum_{j=1}^n |u_j |^2 = n$  excludes   the possibility of singular behavior of each $u_j$ near its zero.
So, if $ ( u_1^*,\cdots, u_n^*) $ is a minimizer of
\[  \sum_{j=1}^n\int_\Om |\nabla u_j|^2  dx \qon \calX(g_1,\cdots,g_n;\Om),
\]
then for a  sequence  of minimizers $(u_{1,\ve},\cdots,u_{n,\ve})$  for \eqref{eq:ftnal G}, we obtain
\begin{equation}
\label{eq:SGL energy bdd}
 \frac12 \int_\Om \sum_{j=1}^n |\nabla u_{j,\ve}|^2 dx + \frac{1}{4\ve^2} \int_\Om \Big(n-\sum_{j=1}^n |u_{j,\ve}|^2\Big)^2 dx \le G_\ve ( u_1^*,\cdots, u_n^*).
\end{equation}
So, the asymptotics of minimizers can be analyzed as in the degree zero case for the classical model \eqref{eq:ftnal E BBH}.
We note that $( u_1^*,\cdots, u_n^*)$ satisfies 
\begin{equation}\label{eq:u1v1 system}
\left\{
\begin{aligned}
-\Delta u_i & = \frac{1}{n}u_i \sum_{j=1}^n |\nabla u_j|^2  \text{ in } \Om, \\
 u_i&=g_i   \text{ on } \pa \Om,\quad
  \sum_{j=1}^n|u_j|^2 =n   \text{ in } \Om.
\end{aligned}
\right.
\end{equation} 
It is an open problem whether \eqref{eq:u1v1 system} allows a unique solution.

As in the equation \eqref{eq:GL weight}, let us introduce a positive weight function  $a_i(x)$ for $u_i$ for each $i$.
So, one may consider the following system: for $i=1,\cdots,n$
\begin{equation}
\label{eq:semilocal GL weight}
\left\{
\begin{aligned}
- \mbox{div}  \big( a_i(x)  \nabla u_i \big) & = \frac{1}{\ve^2} u_i \Big(n-\sum_{j=1}^n |u_j|^2\Big) \qin \Omega,\\
u_i&=g_i  \qon \pa\Omega.
\end{aligned}
\right.
\end{equation}
Given a solution  $(u_{1,\ve},\cdots,u_{n,\ve})$ of \eqref{eq:semilocal GL weight}, if we set  $\tilu_{i,\ve}(x)=u_{i,\ve}(\ve x)$ and $\tila_i(x) = a_i(\ve x)$, then $\tilu_{i,\ve}$'s  satisfy
\begin{equation}
\label{eq:scale SGL}
- \mbox{div}  \big( \tila_i(x)  \nabla \tilu_i \big) =  \tilu_{i,\ve} \Big(n-\sum_{j=1}^n |\tilu_{j,\ve}|^2\Big) \qin \Omega_\ve.
\end{equation}
By taking the limit $\ve \to 0$, we are led to
\begin{equation}\label{eq:SGL-rtwo}
-   \la_i \Delta u_i =  u_i  \Big(n-\sum_{j=1}^n |u_j|^2\Big) \qin \rtwo,
\end{equation}
where   $\la_i=a_i(0)  $.
We recall that   \eqref{eq:BBH-rtwo} can be reduced to \eqref{eq:BBH-rtwo-before-lambda}, i.e.  the case $\la=1$,     by the change of  variables $u(x) \mapsto u(\sqrt{\la} x)$.
However,  if $n>1$, such scale argument does not work anymore for \eqref{eq:scale SGL}.
A role of the constants $\la_1,\cdots,\la_n$ will be addressed briefly in Theorem \ref{thm:main-lambda}.

A  natural question regarding  \eqref{eq:SGL-rtwo} is  whether  the conclusion of Theorem A for the single GL system \eqref{eq:BBH-rtwo}  is still valid for the solutions of the $n$-component GL system \eqref{eq:SGL-rtwo} or not.
There is a big difference in this analysis between the single case \eqref{eq:BBH-rtwo} and the multi-component case \eqref{eq:SGL-rtwo}.
If a solution $u$ of   \eqref{eq:BBH-rtwo} satisfies \eqref{eq:GL-potential}, then one can show two important things
\begin{equation}
\label{eq:property GL}
|u|^2\le 1 \text{ ~in~ } \rtwo \qand  |u(x)|^2 \to 1 \text{ ~as~ } |x| \to \infty.
\end{equation}
These properties play central roles in the proof of Theorem A.
Meanwhile, for a solution pair $(u_1,\cdots,u_n)$ of \eqref{eq:SGL-rtwo} satisfying
\begin{equation}
\label{eq:potential SGL}
E( u_1,\cdots,u_n)=\int_\rtwo \Big(n-\sum_{j=1}^n |u_j|^2\Big)^2 dx<\infty,
\end{equation}
we can show  by Lemma \ref{lem:infty} below  that
\begin{equation}
\label{eq:property SGL}
\sum_{j=1}^n |u_j|^2\le n \text{ ~in~ } \rtwo \qand \sum_{j=1}^n |u_j|^2 \to n \text{ ~as~ } |x| \to \infty.
\end{equation}
Unlike the single case \eqref{eq:property GL}, the properties  \eqref{eq:property SGL} do not provide  any   individual behaviors of $u_j$.
This makes the analysis for solutions of   \eqref{eq:SGL-rtwo} sophisticated and motivates a new research problem.
The purpose of this paper is to verify how such obstruction can be overcome   for a  conclusion similar to Theorem A.
Especially, we focus on the quantization of the potential $E( u_1,\cdots,u_n)$.
Our idea in this paper can be applicable for other types of $n$-component Ginzburg-Landau equations.
We will address this topic elsewhere.

We start with a problem analogous to Theorem A (ii).
In other words, we consider the question whether the equation  \eqref{eq:SGL-rtwo}  allows only trivial solutions whenever  the gradients of $u_j$ are square summable.
Regarding this problem, we obtain the following result.

\begin{theorem}
\label{thm:main1}
Let $( u_1,\cdots,u_n)$ be a solution of \eqref{eq:SGL-rtwo} satisfying
\begin{equation}
\label{eq:pot1}
 \sum_{j=1}^n   \int_\rtwo |\nabla u_j|^2 dx< \infty.
\end{equation}
Then, either $( u_1,\cdots,u_n)\equiv(0,\cdots,0)$  or $( u_1,\cdots,u_n)=(c_1,\cdots,c_n)$ for some  constants $c_1,\cdots, c_n \in \cpx$ with $|c_1|^2+\cdots+|c_n|^2=n$.
\end{theorem}

By Theorem \ref{thm:main1}, if $( u_1,\cdots,u_n)$ is a solution of \eqref{eq:SGL-rtwo} satisfying \eqref{eq:pot1}, then $E( u_1,\cdots,u_n)$ is either $0$ or $\infty$.
Hence, if  a solution $( u_1,\cdots,u_n)$ takes a finite nonzero value of  the potential $E( u_1,\cdots,u_n)$, then $\nabla u_j \not\in L^2(\rtwo)$ for some $j$.
A simple example of such solutions can be constructed from solutions  of \eqref{eq:BBH-rtwo}.
Indeed,   let $U(x)$ be a solution of \eqref{eq:BBH-rtwo} such that $|U(x)| \to 1$.
If we assume $\la_1=\cdots=\la_n=\la$ and set $u_j(x)=U(\sqrt{n}x)$ for each $j$, then $( u_1,\cdots,u_n)$ is a solution of \eqref{eq:SGL-rtwo} with $\nabla u_j \not\in L^2(\rtwo)$.

Another example is a radially symmetric vortex solution.
If we put $u_i(x)=e^{d_i \theta}f_i(r)$ where $x=re^{i\theta}$ and $d_i$ is a nonnegative integer, then  we can rewrite \eqref{eq:SGL-rtwo} as 
\begin{equation}
\label{eq:SGL-rtwo-radial}
\left\{
\begin{aligned}
& f_i''+\frac{1}{r} f_i' - \frac{d_i^2}{r^2} f_i + \frac{1}{\la_i} f_i \Big( n -\sum_{j=1}^n |f_j|^2 \Big)=0\qfor r>0,\\
& f_i(r) = b_i  r^{d_i} + o(r^{d_i}) \quad \mbox{for some $b_i>0$} \qas r \searrow 0.
\end{aligned}\right.
\end{equation} 
For the case $n=1$, by letting $u(r) = e^{d\theta}f(r)$, we can transform \eqref{eq:BBH-rtwo} into
\begin{equation}
\label{eq:GL-rtwo-radial}
\left\{
\begin{aligned}
& f ''+\frac{1}{r} f  - \frac{d^2}{r^2} f + \frac{1}{\la} f  ( 1 -f  )=0\qfor r>0,\\
& f (r) = b   r^{d } + o(r^{d }) \quad \mbox{for some $b>0$} \qas r \searrow 0.
\end{aligned}\right.
\end{equation} 
The existence and properties of solutions to \eqref{eq:GL-rtwo-radial} are well-known in \cite{CEQ94,HH94}.
In particular, there exists a unique   solution  $f  $ such that
\begin{equation}
\label{eq:radial n=1}
f  (r) \nearrow 1 \qas r\to \infty,
\end{equation}
 and $u  $ satisfies the quantization property of Theorem A.
When $n>1$ and $(f_1,\cdots,f_n)$ is a solution of \eqref{eq:SGL-rtwo-radial}, then the condition \eqref{eq:radial n=1} can be generalized by
\begin{equation}
\label{eq:radial n>1 asympt}
 \text{$f_i (r) \nearrow \al_i$ for some $\al_i\in (0,n)$ with $\al_1^2+\cdots+\al_n^2=n$.}
\end{equation}
Since one may consider infinitely many choices of $(\al_1,\cdots,\al_n)$  whereas $\al_1=1$ is the unique choice for $n=1$, it is expected that the solution structures  of \eqref{eq:SGL-rtwo-radial} is richer than   those of     \eqref{eq:GL-rtwo-radial}.
We will address this topic  elsewhere.
Instead, we concentrate on the quantization of potential $E(u_1,\cdots,u_n)$ in this paper.

Before proceeding to the quantization problem, it is worthwhile to see that   the constants $\la_1,\cdots,\la_n$ play an important role to get nontrivial solutions of \eqref{eq:SGL-rtwo-radial}.
We have seen that if $f$ is the unique solution  of \eqref{eq:GL-rtwo-radial}, then $(f_1,\cdots,f_n)$ with $f_j(r)=f(\sqrt{n} r)$ for each $j$ is a solution of \eqref{eq:SGL-rtwo-radial} when $d_1=\cdots=d_n=d$ and $\la_1=\cdots=\la_n =\la$.
The next theorem tells us that if $\la_1=\cdots=\la_n$, then all $d_j$'s must be equal.
So, given a solution $(f_1,\cdots,f_n)$ of \eqref{eq:SGL-rtwo-radial}, there is a possibility that each $f_j$ is a scale of a solution of \eqref{eq:GL-rtwo-radial}.
Therefore, it is necessary to assume that all $\la_j$'s  are not equal in order to get a nontrivial solution of \eqref{eq:SGL-rtwo-radial} in the sense that it doesn't come from $n$-copies of solutions of \eqref{eq:GL-rtwo-radial}.
\begin{theorem}
\label{thm:main-lambda}
Let $( f_1,\cdots,f_n)$ be a solution of \eqref{eq:SGL-rtwo-radial}   that satisfies \eqref{eq:radial n>1 asympt}.
 If $\la_1=\cdots=\la_n $,  then $d_1=\cdots=d_n$.
\end{theorem}

To state the quantization effect, we need a notation.
For a nonnegative real number $\al$ and a smooth function $u:\rtwo \to \cpx$, we set
\begin{equation}
\label{eq:def of I-alpha}
I_\al (u) = \int_\rtwo (\al^2  - |u|^2)^2 dx  .
\end{equation}
The third result of this paper is the following.

\begin{theorem}
\label{thm:main2}
Let $( u_1,\cdots,u_n)$ be a solution of \eqref{eq:SGL-rtwo} such that
\begin{equation}
\label{eq:E finite}
E( u_1,\cdots,u_n)<\infty.
\end{equation}
Suppose that there are   nonnegative real numbers $\al_1,\cdots,\al_n$ such that
\begin{equation}
\label{eq:I-alpha finite}
\al_1^2+\cdots+\al_n^2=n \qand I_{\al_j} (u_j)<\infty.
\end{equation}
Then, there exist  nonnegative integers $d_j \in \{0,1,2,\cdots \}$ for $1\le j\le n$ such that
\begin{equation}
\label{eq:E-quant}
E( u_1,\cdots,u_n) = 2\pi \sum_{j=1}^n  \la_j \al_j^2 d_j^2.
\end{equation}
\end{theorem}

The basic strategy for the proof of  Theorem \ref{thm:main1} and Theorem \ref{thm:main2} is based on the proof of Theorem A given in \cite{BMR94}.
However, there arise many nontrivial situations in the proof of Theorem \ref{thm:main2} that do not appear in the analysis of solutions of \eqref{eq:BBH-rtwo}.
For instance, any nontrivial solution $u$ of  \eqref{eq:BBH-rtwo} with \eqref{eq:GL-potential} satisfies \eqref{eq:property GL}.
Whereas if $( u_1,\cdots,u_n)$ is a solution of \eqref{eq:SGL-rtwo}, then $|u_1|^2+\cdots+|u_n|^2 \to n$ but there is no information on the behavior of individual $u_j$.
Moreover,  the maximum principle only gives the inequality $|u_1|^2 +\cdots + |u_n|^2 \le n$ but   do not  provide any pointwise estimate of individual $u_j$.
These properties make our problem difficult and require  some new ideas.
 Throughout  a detailed analysis, we will deduce why the condition \eqref{eq:I-alpha finite} is   reasonable  and necessary in the proof.

The rest of this paper is organized as follows.
In Section 2, we prove Theorem \ref{thm:main1} and Theorem \ref{thm:main-lambda}.
In Section 3, we prove  Theorem \ref{thm:main2}.
The main idea for the proofs of Theorem \ref{thm:main1} and Theorem \ref{thm:main2} is to derive Pohozaev identities since we will use  finiteness conditions of functional-like quantities such as \eqref{eq:pot1} and \eqref{eq:E finite}.
Meanwhile, Theorem \ref{thm:main-lambda} is proved by a technique of  integration by parts for the ODE system \eqref{eq:SGL-rtwo-radial}.

We close this section by introducing some notations.
We denote $\bar{a}$ by the complex conjugate of $a \in \cpx$.
We write $B_R=B(0,R)$ and $S_R=\pa B_R$.
We set $\xi_k(x)=\xi(x/k)$ where   $\xi\in C_c^\infty(\rtwo)$ such that $0\le\xi\le 1$, $\xi\equiv1$ for $|x|\le 1$, and $\xi=0$ for $|x|\ge 2$.
By the symmetry property of \eqref{eq:SGL-rtwo}, we may assume that
\begin{equation}
\label{eq:oder lambda-j}
\la_1  \le \la_2 \le  \cdots \le \la_n.
\end{equation}


\section{Proof of Theorem \ref{thm:main1} and Theorem \ref{thm:main-lambda}}\label{sec:d2>=d1=0}
\setcounter{equation}{0}

This section is devoted to the proof of Theorem \ref{thm:main1}  and Theorem \ref{thm:main-lambda}.
Throughout this section, we put
\begin{equation}
\label{eq:f h def}
f= \sum_{j=1}^n  |u_j|^2 \qand h =  \sum_{j=1}^n \frac{1}{\la_j} |u_j|^2.
\end{equation}
We have $\la_n^{-1}  f \le h \le \la_1^{-1} f$ by \eqref{eq:oder lambda-j}.
We often use the following identities:
\begin{equation}
\label{eq:vp0}
\left\{
\begin{aligned}
-\la_i \Delta u_i& =   u_i(n-f),\\
 - \Delta f &= 2 h  (n-f)   -2 \sum_{j=1}^n |\nabla u_j|^2.
 \end{aligned}
 \right.
\end{equation}

\begin{lemma}\label{lem:infty-nabla}
If  $(u_1,\cdots,u_n)$ is  a solution   of  \eqref{eq:SGL-rtwo} satisfying \eqref{eq:pot1}, then
$f= |u_1|^2+\cdots+|u_n|^2 \le n $ in $\rtwo$.
Moreover, $\nabla u_j \in L^\infty(\rtwo)$ for each $j$.
\end{lemma}
\begin{proof}
Let $\vp=\sqrt{f}-\sqrt{n} $.
By using \eqref{eq:vp0}, we derive
\begin{align*}
\Delta \vp&= \vp h f^{- \frac12}  (\sqrt{f}+\sqrt{n})
 -\frac{1}{4} f^{-\frac32} \Big(  |\nabla f|^2 -4  f  \sum_{j=1}^n   |\nabla u_j|^2  \Big)\\
 &\ge \vp h f^{- \frac12}  (\sqrt{f}+\sqrt{n})  ,
\end{align*}
where the inequality comes from the Cauchy-Schwartz inequality:
\[  |\nabla f|^2 = 4 \Big(  \sum_{j=1}^n \re \big[ \bar{u}_j \nabla u_j \big]  \Big)^2 \le 4  f  \sum_{j=1}^n   |\nabla u_j|^2.
\]
 Hence, by Kato's inequality and \eqref{eq:oder lambda-j}
\begin{equation}
\label{eq:Kato}
 \Delta \vp^+ \ge  \chi_{ \{\vp>0\} }\Delta \vp \ge  \frac{1}{\la_n} \vp^+\sqrt{f}(\sqrt{f}+\sqrt{n}) \ge  \frac{2n}{\la_n} \vp^+.
\end{equation}
Again by the  Cauchy-Schwartz inequality and \eqref{eq:pot1},
\[ |\nabla \vp^+ | = \frac{1}{\sqrt{f}}  \Big| \sum_{j=1}^n \re \big[ \bar{u}_j \nabla u_j \big] \Big| \le \Big( \sum_{j=1}^n   | \nabla u_j |^2 \Big)^{\frac12} ~\in ~ L^2(\rtwo).
\]
So, multiplying \eqref{eq:Kato} by $\xi_k$, we can see that
\begin{align*}
 \frac{2n}{\la_n} \int_\rtwo \xi_k  \vp^+  & \le -\int_\rtwo \nabla \vp^+ \cdot \nabla \xi_k \le \frac{C}{k} \int_{\{ k<|x|<2k\} } |\nabla \vp^+| \\
 &\le C \Big( \int_{\{ k<|x|<2k\} } |\nabla \vp^+|^2 \Big)^{\frac12} \to 0
\end{align*}
as $k\to \infty$.
Hence, $\vp^+\equiv 0$.

Now, we note that $|\Delta u_j|\le \la_j^{-1}  |u_j| (n-f)\le  \la_j^{-1} n\sqrt{n}$.
Thus, for  each $x \in \rtwo$, the function $v_{j,x}(y)=u_j(x+y)$ for $y\in B_1(0)$ is uniformly bounded in $W^{2,p}(B_1)$ for all $p>1$.
Hence, $\nabla v_{j,x}$ is uniformly bounded in $C^1(B_{1/2})$ and thus  $\nabla u_j \in L^\infty(\rtwo)$.
\end{proof}

{\bf Proof of Theorem \ref{thm:main1}.}
By multiplying the second equation of  \eqref{eq:vp0} by  $\xi_k$, we obtain
\begin{equation}\label{eq:3.1}
2 \int_\rtwo \xi_k h (n-f) =   \int_\rtwo \nabla \xi_k \cdot  \nabla f  + 2 \sum_{j=1}^n \int_\rtwo \xi_k  |\nabla u_j|^2 .
\end{equation}
Since $u_j \in L^\infty (\rtwo)$  and $\nabla u_j \in L^\infty(\rtwo)$, it follows from  \eqref{eq:pot1} that
\[ \Big| \int_\rtwo \nabla \xi_k \cdot  \nabla f   \Big| \le C  \sum_{j=1}^n  \Big( \int_{k \le |x| \le 2k}  | \nabla u_j|^2 \Big)^{\frac12} \to 0
\]
as $k\to \infty$.
Thus, letting $k\to \infty$, we are led to
\begin{equation}\label{eq:3.2}
\frac{1}{\la_n}  \int_\rtwo f (n-f)  \le  \int_\rtwo h (n-f) = \sum_{j=1}^n \int_\rtwo |\nabla u_j|^2 .
\end{equation}
Since $\nabla u_j \in L^\infty(\rtwo)$, this implies that the set $\{ x:  1\le f (x)\le 3/2<n \}$ is bounded.
Thus, we can conclude from \eqref{eq:3.2} that either
$f \in L^1(\rtwo)$ or $(n-f) \in L^1(\rtwo)$.

Meanwhile, by multiplying \eqref{eq:SGL-rtwo} by $\xi_k x\cdot \nabla u_i$,  we obtain
\begin{align*}
(LHS)&=- \sum_{i=1}^n \int_\rtwo \la_i \xi_k  \Delta u_i(x\cdot \nabla u_i) \\
& = \sum_{i=1}^n \int_\rtwo \xi_k u_i (f-n)(x\cdot \nabla u_i)=(RHS).
\end{align*}
We have
\begin{align*}
(LHS)&=  \sum_{i=1}^n  \la_j \int_\rtwo (\nabla \xi_k \cdot \nabla u_i)(x\cdot \nabla u_i) -  \sum_{i=1}^n \frac{ \la_i}{2} \int_\rtwo |\nabla u_i|^2 (x\cdot \nabla \xi_k).
\end{align*}
So,
\begin{align*}
 |(LHS)|\le C  \sum_{i=1}^n \int_{\{ k<|x|<2k\} } |\nabla u_i|^2 \to  0 \qas k\to \infty.
\end{align*}

Now, suppose that $f \in L^1(\rtwo)$.
Then, since $f \le n$ and $f \in L^1(\rtwo)$,
\begin{align*}
(RHS)&=  \frac12  \int_\rtwo  \xi_k ( x \cdot \nabla f ) (f-n) \\
& =  \int_\rtwo \xi_k f \Big( n -\frac{f}{2} \Big) + \frac12 \int_\rtwo     f \Big( n -\frac{f}{2} \Big) (x\cdot \nabla \xi_k)\\
&\ge  \frac{n}{2}   \int_\rtwo \xi_k f  +o(1)
\end{align*}
as $k\to \infty$.
This implies that $f \equiv 0$ and thus $u_i \equiv 0$ for each $i$. 
On the other hand, if $(n-f) \in L^1(\rtwo)$, then
\begin{align*}
(RHS)&=  \frac14  \int_\rtwo  \xi_k   x \cdot \nabla  (f-n)^2 \\
& =-  \frac12 \int_\rtwo \xi_k (f-n)^2  - \frac14 \int_\rtwo (x\cdot \nabla \xi_k)    ( f-n)^2 \\
&=-  \frac12 \int_\rtwo \xi_k (f-n)^2 +o(1)
\end{align*}
as $k\to \infty$.
So, we have  $f \equiv n$, which implies by \eqref{eq:3.1}  that $\nabla u_j\equiv 0$ for each $j$.
This gives us the desired conclusion of Theorem \ref{thm:main1}.
\qed \\

{\bf Proof of Theorem \ref{thm:main-lambda}.}
Let us assume that $\la_1=\cdots=\la_n =\la$.
For instance, suppose that $d_1>d_2$.
Then, we have from \eqref{eq:SGL-rtwo-radial}
  \begin{align*}
& f_1''+\frac{1}{r} f_1' =  \frac{d_1^2}{r^2} f_1+ \frac{1}{\la } f_1 \Big( \sum_{j=1}^n f_j^2-n \Big), \\
&  f_2''+\frac{1}{r} f_2' =  \frac{d_2^2}{r^2} f_2+ \frac{1}{\la } f_2 \Big(  \sum_{j=1}^n f_j^2-n \Big).
\end{align*}
Choose $r_0>0$ such that $f_1 \ge \al_1 /2$ and $f_2 \ge \al_2 /2$ for $r\ge r_0$.
So,  for all $r \ge r_0$
\[ \frac{1}{r} \big[ r(f_1' f_2 - f_1 f_2') \big]' =  \frac{d_1^2 - d_2^2}{r^2} f_1 f_2 \ge  \frac{\al_1 \al_2(d_1^2 - d_2^2)}{4} \cdot \frac{1}{r^2}  \equiv \frac{c_0}{r^2} .
\]
Integrating this inequality on $(r_0,r)$, we are led to
\[ f_1' f_2 - f_1 f_2' \ge \frac{c_0}{r} \ln r + \frac{c_1}{r}
\]
for some $c_1 \in \rone$.
Set $h=f_1/f_2$.
Then, $h(r) \to \al_1/\al_2$ as $r\to \infty$ and
\[ h' = \frac{f_1' f_2 - f_1 f_2' }{f_2^2} \ge \frac{1}{\al_2^2} (f_1' f_2 - f_1 f_2' )\ge \frac{c_0}{\al_2^2r} \ln r + \frac{c_1}{\al_2^2r} \qfor r\ge r_0.
\]
This leads us to a contradiction: as $r\to \infty$,
\[ h(r) \ge \frac{c_0}{2\al_2^2 }  \big[ (\ln r)^2 - (\ln r_0)^2 \big] + \frac{c_1}{\al_2^2 } \ln \frac{r}{r_0} + h(r_0) ~~\to ~~\infty.
\]
Consequently, $d_1=d_2$.
\qed \\

\section{Proof of Theorem \ref{thm:main2}}
\setcounter{equation}{0}

In this section, we prove Theorem \ref{thm:main2}.
The key part is to classify  the values of potentials $I_{\al_i}(u_i)$   into four cases.
Throughout this section, let  $(u_1,\cdots,u_n)$ be a solution of  \eqref{eq:SGL-rtwo} satisfying
\begin{equation}\label{eq:pot_n}
E=E(u_1,\cdots,u_n)=\int_\rtwo \Big(n-\sum_{j=1}^n|u_j|^2\Big)^2<\infty.
\end{equation}
We begin with the following lemma.

\begin{lemma}\label{lem:infty}
Let $(u_1,\cdots,u_n)$ be  any solution   of  \eqref{eq:SGL-rtwo} satisfying \eqref{eq:pot_n}. 
Then,  the followings hold.
\begin{itemize}
\item[{\rm (i)}]
We have
\begin{equation}\label{eq:behavior u_j}
 \sum_{j=1}^n|u_j|^2 \le n \qand \sum_{j=1}^n|u_j|^2 \to n \qas    |x|\to\infty.
\end{equation}
\item[{\rm (ii)}]
We have
\begin{equation} \label{eq:gradient L-infty}
 \nabla u_j\in L^\infty(\rtwo)~ \mbox{ for each } ~j.
\end{equation}
\item[{\rm (iii)}]
There exists a constant $C$ independent of $R$ such that
\begin{align}
 \label{eq:int}
 \sum_{j=1}^n \int_{B_R} |\nabla u_j|^2 dx \le CR.
\end{align}
\end{itemize}
\end{lemma}
\begin{proof}
(i)
Let $f=|u_1|^2+\cdots+|u_n|^2$ and $ \vp=f-n $.
By    \eqref{eq:vp0},
\[  \Delta \vp \ge  2 h \vp .
\]
Multiplying this equation by $\xi_k\vp^+$, we are led to
\[\begin{aligned}
& \int_\rtwo \big(|\nabla\vp^+|^2    +2h  |\vp^+|^2\big) \xi_k  dx
\le    \frac{1}{2}\int_\rtwo\Delta\xi_k|\vp^+|^2 dx\le \frac{C}{k^2}\int_\rtwo|\vp^+|^2 dx.
\end{aligned}
\]
Letting $k\to\infty$, we see by \eqref{eq:pot_n} that $\vp^+\equiv 0  $.
This implies the first part of \eqref{eq:behavior u_j}.

Next, to show the second part of \eqref{eq:behavior u_j}, we argue by a contradiction.
Assume that there exists a sequence $|x_k|\to\infty$ such that $f(x_k)\le n-2\delta$ for some $0<\delta<2/5$.
Let  $r= \min\big\{1, \delta/(2Mn\sqrt{n}+nM^2)\big\}$, where we denote $M=\max_j \|\nabla u_j\|_{L^\infty}$ in the rest of this paper.
For $x\in B (x_k ,r)$, we have by the Mean Value Theorem
\[\begin{aligned}
 f(x) & \le \sum_{j=1}^n \big( |u_j(x_k)|+|u_j(x)-u_j(x_k)|\big)^2  \\
&\le   \sum_{j=1}^n |u_j (x_k)|^2+2M|x_k-x| \sum_{j=1}^n |u_j(x_k)|+nM^2|x_k-x|^2\\
&<  n-2\delta +(2Mn\sqrt{n}+nM^2)r \le n- \delta.
\end{aligned}
\]
Then, it follows that
\[\int_{B (x_k ,r)}  (f-n)^2 dx\ge \pi r^2 \delta^2.
\]
This is a contradiction since  $f-n \in L^2(\rtwo)$  by \eqref{eq:pot_n}. 

(ii)
This follows from Lemma \ref{lem:infty-nabla}.

(iii)
By multiplying   \eqref{eq:SGL-rtwo} by  $u_i$ for each $i$  and
  integrating them over $B_R$, we obtain
\begin{align*}
  \sum_{i=1}^n \int_{B_R} \la_i |\nabla u_i|^2 dx
&=   \sum_{i=1}^n \int_{\partial B_R} \la_i   u_i \frac{\partial u_i}{\partial \nu} dS+\int_{B_R}f (n-f) dx \\
& \le   n \sqrt{n} \la_n  M |\partial B_R| + n E^{\frac12} | B_R|^{\frac12} \le CR.
\end{align*} 
Here $\nu$ denotes the outward normal to $B_R$.
\end{proof}

We note that although the statement  $ |u_1|^2+\cdots+|u_n|^2 \le n $   in Lemma \ref{lem:infty-nabla} and in \eqref{eq:behavior u_j} is the same, it is proved under different conditions.
The following lemma helps us classify the values of  $I_{\al_j}(u_j)$.

\begin{lemma}
\label{eq:I_j infinite number}
Let $n\ge 3$ and $(u_1,\cdots,u_n)$ be a solution of  \eqref{eq:SGL-rtwo}  satisfying \eqref{eq:pot_n}.
For $1\le j \le n$, let  $\al_j \in \big[0,\sqrt{n}\,\big]$ be given such that $\al_1^2+\cdots+\al_n^2=n$.
If  $k \in \{ 0,1,\cdots,n\}$ is the number of $\al_j$ such that $I_{\al_j}(u_j)=\infty$, then   $k\ne 1$.
Moreover, if $I_{\al_j}(u_j)<\infty$, then
\begin{equation}\label{eq: u_j alpha_j}
 |u_j|  \to \al_j  \quad \text{uniformly} \quad  \text{as}\quad   |x|\to \infty . 
\end{equation}
\end{lemma}
\begin{proof}
Suppose that $I_{\al_j}(u_j)<\infty$ for $1\le j\le n-1$.
Let
\begin{align*}
X_R &= \Big[ \int_{B_R}  \Big( \sum_{j=1}^{n-1} \al_j^2  - \sum_{j=1}^{n-1}|u_j|^2\Big)^2dx \Big]^{\frac12},\\
  Y_R& =  \Big[\int_{B_R}  ( \al_n^2 - |u_n|^2 )^2dx \Big]^{\frac12} .
\end{align*}
Since $I_{\al_j}(u_j)<\infty$ for $1\le j\le n-1$,  it follows that  $X_R\le C$ where $C$ is independent of $R$.
By \eqref{eq:pot_n}  and  the  Cauchy-Schwartz inequality,
\[\begin{aligned}
\infty>E&=\int_\rtwo \Big(n-\sum_{j=1}^n|u_j|^2\Big)^2 dx \\
& \ge X_R^2+Y_R^2+2\int_{B_R}  \Big( \sum_{j=1}^{n-1} \al_j^2  - \sum_{j=1}^{n-1}|u_j|^2\Big) ( \al_n^2 - |u_n|^2 ) dx \\
& \ge  (X_R - Y_R )^2.
\end{aligned}
\]
As a consequence,   $Y_R$ is also uniformly bounded as $R\to \infty$ and thus $I_{\al_n}(u_n)<\infty$.
This   implies that  either $k=0$ or $2\le k\le n $.
The behavior \eqref{eq: u_j alpha_j} follows from the same argument for the proof of (i) in  Lemma \ref{lem:infty}.
\end{proof}

By   Lemma \ref{eq:I_j infinite number}, if $(u_1,\cdots,u_n)$ is a solution of  \eqref{eq:SGL-rtwo} satisfying  \eqref{eq:pot_n}, then  one of the following four cases holds:
\begin{itemize}
\item[{\rm\bf (P1)}]
 there exist  $\alpha_j  \in [0,\sqrt{n}\,]$ for $1\le j\le n$ such that
\[ \al_j>0 ~~\forall j, \quad \sum_{j=1}^n \al_j^2  =n, \qand I_{\al_j}(u_j)<\infty  ~~\forall j ;
\]
\item[{\rm\bf (P2)}]
there exist $l\in \{1,\cdots,n-1\}$ and $\alpha_j  \in [0,\sqrt{n}\,]$ for $1\le j\le n$ such that
\[ \left\{
\begin{aligned}
& \al_j>0 ~\text{ for }~ 1\le j \le l, \quad   \al_j=0 ~\text{ for }~ l+1\le j\le n,\\
& \sum_{j=1}^n \al_j^2  =n, \qand I_{\al_j}(u_j)<\infty  ~\forall j ;
\end{aligned}
\right.\]
\item[{\rm\bf (P3)}]
there exist  $l\in \{1,\cdots,n-2\}$ and $\alpha_j  \in[0,\sqrt{n}\,]$ for $1\le j\le n$  such that
\[ \left\{
\begin{aligned}
& I_{\al_j}(u_j)<\infty ~\text{ for }~ 1\le j \le l, \quad I_{\al_j}(u_j)=\infty ~\text{ for }~ l+1\le j\le n,\\
& \sum_{j=1}^n \al_j^2  =n;
\end{aligned}
\right.\]
\item[{\rm\bf (P4)}]
$I_{\al_j}(u_j)=\infty $ for any choice of  $\alpha_j \in[0,\sqrt{n}\,]$ satisfying $\al_1^2+\cdots+\al_n^2=n$.\\
\end{itemize}

\begin{remark}\label{rmk:P-classify}$\;$
\begin{itemize}
\item[{\rm\bf (a)}]
 In view of Lemma \ref{eq:I_j infinite number}, $l\ne n-1$ in the case {\rm\bf (P3)}.
 In particular, {\rm\bf (P3)} does not happen for $n=2$.

\item[{\rm\bf (b)}]
  Let us explain the above classification for the case $n=3$.
Suppose that $\al_1,\al_2,\al_3 \in [0,\sqrt{3}\,]$ and $\al_1^2+\al_2^2+\al_3^2=3$.
Assume that $I_{\al_1}(u_1)<\infty$.
By the same argument of Lemma \ref{eq:I_j infinite number}, it is obvious that
\[I_{\al_1}(u_1)<\infty \quad\mbox{if and only if}\quad J=\int_\rtwo \big(\al_2^2+\al_3^2-|u_2|^2-|u_3|^2\big)^2<\infty.
\]
When $J<\infty$, we have that $I_{\al_2}(u_2)<\infty$ if and only if $I_{\al_3}(u_3)<\infty$.
So, if $ I_{\al_1}(u_1)<\infty$, then either
\[\text{(i) } I_{\al_1}(u_1)<\infty, ~I_{\al_2}(u_2)<\infty, ~I_{\al_3}(u_3)<\infty,
\]
or
\[\text{(ii) } I_{\al_1}(u_1)<\infty, ~I_{\al_2}(u_2)=\infty, ~I_{\al_3}(u_3)=\infty.
\]
If (i) is true and $\al_j>0$ for all $j=1,2,3$, then it results in the case {\bf (P1)}.
If (i) is true and either $\al_1,\al_2>0$ with $ \al_3=0 $ or $\al_1=\sqrt{3}$ with $\al_2=\al_3=0$, then we are led to the case {\bf (P2)}.
If (ii) is true, then we have {\bf (P3)}.
Finally, if $I_{\al_1}(u_1)=I_{\al_2}(u_2)=I_{\al_3}(u_3)=\infty$ for any choice of $\al_j$ with $\al_1^2+\al_2^2+\al_3^2=3$, then we get the case {\bf (P4)}.

\item[{\rm\bf (c)}]
We believe that  no solutions of  \eqref{eq:SGL-rtwo} satisfy both  $E( u_1,\cdots,u_n)<\infty$  and {\bf (P3)} (or {\bf (P4)}).
Of course, one can find an example of $(u_1,\cdots,u_n)$ that violates one of three conditions: (i)   $(u_1,\cdots,u_n)$ is  a solution, (ii) $E( u_1,\cdots,u_n)<\infty$,   and (iii) {\bf (P3)}  or {\bf (P4)} is true.

For instance,  if $n\ge 2$ and we set
\begin{align*}
u_1(x,y)&=\sqrt{n}\sin (x^2+y^2),\\
 u_j(x,y)& =\sqrt{n}\cos^{j-1} (x^2+y^2) \sin  (x^2+y^2) \qfor 2\le j \le n-1,\\
 u_n(x,y)& =\sqrt{n}\cos^{n-1} (x^2+y^2),
\end{align*}
then $(u_1,\cdots,u_n)$ satisfies $E( u_1,\cdots,u_n)<\infty$  and {\bf (P4)}.
However, $(u_1,\cdots,u_n)$ is not a solution of  \eqref{eq:SGL-rtwo}.

Another example is   $u_j(x ,y) = A_j e^{i\om x }$ where $A_j$ and $\om $  are positive real numbers for $1\le j \le n$ such that $A_1^2+\cdots+A_n^2+\om^2=n$.
Then, one may check that  $(u_1,\cdots,u_n) $ is a solution of \eqref{eq:SGL-rtwo}.
Moreover, if $\al_j=A_j$ for $1\le j\le n-1$ and  $\al_n=\sqrt{A_n^2+\om^2}$,  then $I_{\al_j}(u_j)=0$ for $1\le j\le n-1$  and $I_{\al_n}(u_n)=\infty$.
Thus, $( u_1,\cdots,u_n)$ satisfies {\bf (P3)} with $l=n-1$ but we obtain $E( u_1,\cdots,u_n)=\infty$.
 \qed
\end{itemize}
\end{remark}

Based on the above classification and Remark \ref{rmk:P-classify}, we will focus on the quantization problem only for the cases  {\bf (P1)} and {\bf (P2)}.
Thus,    the condition \eqref{eq:I-alpha finite} in Theorem \ref{thm:main2} is quite a reasonable assumption.
Now, Theorem \ref{thm:main2} is a consequence of the following proposition.

\begin{proposition}\label{prop:quant(P1)_n}
Let $(u_1,\cdots,u_n)$ be a solution pair of \eqref{eq:SGL-rtwo}  satisfying \eqref{eq:pot_n}.
Suppose {\rm\bf (P1)} or {\rm\bf (P2)} is true.
That is, there exist $l\in \{1,\cdots,n \}$ and $\alpha_i  \in [0,\sqrt{n}\,]$ for $1\le i\le n$ such that
\[ \left\{
\begin{aligned}
& \al_i>0 ~\text{ for }~ 1\le i \le l, \\
&   \al_i=0 ~\text{ for }~ l+1\le i\le n \quad \mbox{if} \quad l<n,\\
&  \al_1^2  +\cdots+\al_n^2 =n, \\
&  I_{\al_i}(u_i)<\infty  ~\forall i.
\end{aligned}
\right.\]
Then, there exist $l$ nonnegative integers $d_1,\cdots,d_l$  such that
\begin{equation}
\label{eq:quant P3_n}
\int_\rtwo \Big(\sum_{j=1}^n|u_j|^2-n\Big)^2 dx =2\pi \sum_{i=1}^l \la_i \al_i^2 d_i^2.
\end{equation}
\end{proposition}
\begin{proof}
We follow the argument of \cite{BMR94} that are based on the Pohozaev identities.
Let us choose small $\delta \in (0,1)$ and a number $\mu>1$ such that
\begin{align*}
& 0< \delta < \min \Big\{ \frac{\al_i^2}{2} , \frac14 ~\Big|~1\le i\le l\Big\},\\
& \max \Big\{ \frac{\al_i^2}{\al_i^2-\delta} ~\Big|~1\le i\le l\Big\}  < 2(1-\delta),\\
&   \max \Big\{\frac{\al_i^2}{\al_i^2-\delta} ~\Big|~1\le i\le l  \Big\} <\mu < 2(1-\delta).
\end{align*}
By \eqref{eq: u_j alpha_j}, there exists $R_0>0$ such that for all $|x| \ge R_0$  and $1\le i\le l$, 
\begin{equation}\label{eq:behavior on inf two}
 \al_i^2-\delta < |u_i(x)|^2 < \mu (\al_i^2-\delta) .
\end{equation}
For $R>R_0$, the degrees
\[d_i=\deg (u_i,S_R) \qfor 1\le i\le l
\]
are well-defined.
We notice that given  a solution $(u_1,\cdots,u_n)$ of \eqref{eq:SGL-rtwo}, if we replace some $u_j$ by $\bar{u}_j$, then it gives also a solution of \eqref{eq:SGL-rtwo}.
For instance, if $n=2$ and $(u_1,u_2)$ is a solution, then
 $(\bar{u}_1,u_2)$, $(u_1,\bar{u}_2 )$, and $(\bar{u}_1 ,\bar{u}_2 )$ are also solutions.
So, without loss of generality, we may assume that $d_j\ge 0$ for $1\le j\le l$.
In addition, there exist   smooth real valued functions $\psi_j(x)$  on $\rtwo \setminus B_{R_0}$ such that
\begin{equation}\label{eq:express}
u_j(x) =|u_j(x)|e^{i(d_j\theta +\psi_j)} \equiv \rho_j(x)e^{i\vp_j} \qon \rtwo \setminus B_{R_0} \qfor 1\le j \le l.
\end{equation}
We note that $\psi_j$ is a globally defined smooth function whereas $\vp_j$ is locally defined.

We claim that
\begin{align}
\label{eq:nabla psi bdd}
 \sum_{i=1}^l  \int_{\rtwo\setminus   B_{R_0} } |\nabla \psi_i|^2  dx & <\infty,\\
\label{eq:nabla rho bdd}
  \sum_{i=1}^l  \int_{\rtwo\setminus   B_{R_0} } |\nabla\rho_i|^2dx &<\infty,\\
  \label{eq:remain}
 \sum_{i=l+1}^n  \int_{\rtwo  }|\nabla u_i|^2& < \infty.
\end{align}
We postpone the proof of these estimates to the end of this proof.

By  the Pohozaev identity for \eqref{eq:SGL-rtwo}, i.e., multiplying  $x\cdot \nabla u_i$  on \eqref{eq:SGL-rtwo}, we deduce that for   $r>0$,
\begin{equation}\label{eq:Poho}\begin{aligned}
&~\sum_{i=1}^n\int_{S_r} \la_i \Big|\frac{\partial u_i}{\partial\nu}\Big|^2+\frac{1}{r}\int_{B_r}\Big(n-\sum_{j=1}^n|u_j|^2\Big)^2\\
=&~\sum_{i=1}^n\int_{S_r} \la_i\Big|\frac{\partial u_i}{\partial\tau}\Big|^2+\frac{1}{2}\int_{S_r}\Big(n-\sum_{j=1}^n|u_j|^2\Big)^2.
\end{aligned}
\end{equation}
Set
\[ E(r)=\int_{B_r}\Big(\sum_{j=1}^n|u_j|^2-n\Big)^2  .
\]
Then, $E(r) \to E$ as $r\to \infty$.
By	 integrating  \eqref{eq:Poho} for $r\in(0,R)$, we see that
\begin{equation}\label{eq:Poho int}
 \sum_{i=1}^n\int_{B_R} \la_i\Big|\frac{\partial u_i}{\partial\nu}\Big|^2 +\int_0^R\frac{E(r)}{r}dr
= \sum_{i=1}^n\int_{B_R}\la_i \Big|\frac{\partial u_i}{\partial\tau}\Big|^2+\frac{1}{2}E(R).
\end{equation}
By \eqref{eq:nabla psi bdd} and \eqref{eq:nabla rho bdd}, it holds that
\[ \sum_{i=1}^l \int_{B_R} \la_i \Big|\frac{\partial u_i}{\partial\nu}\Big|^2 \le \sum_{i=1}^l \int_{B_R}\big(\la_i |\nabla\rho_i|^2+\la_i\rho_i^2|\nabla\vp_i|^2 \big)  \le C.
\]
 Moreover, it follows from \eqref{eq:remain} that
\[\sum_{i=l+1}^n \Big(\int_{B_R} \la_i \Big|\frac{\partial u_i}{\partial\nu}\Big|^2 +\int_{B_R} \la_i \Big|\frac{\partial u_i}{\partial\tau}\Big|^2\Big)\le C.
\] 
Hence, by dividing \eqref{eq:Poho int} by $\log R$ and letting $R\to \infty$, we are led to
\begin{equation}\label{eq:E final step}
E = \sum_{i=1}^l\lim_{R\to \infty} \frac{1}{\log R} \int_{B_R}  \la_i \Big|\frac{\partial u_i}{\partial\tau}\Big|^2 .
\end{equation}
For $R>R_0$,  by letting $A_R=B_R\backslash B_{R_0}$, we have
 \begin{align*}
&  \int_{A_R} \Big|\frac{\partial u_i}{\partial\tau}\Big|^2
=  \int_{A_R}  \Big|\frac{\partial \rho_i}{\partial\tau}\Big|^2 + \int_{A_R}  \rho_i^2 \Big| \frac{d_i}{r} + \frac{\partial \psi_i}{\partial \tau}  \Big|^2 \\
 =&   \int_{A_R} \Big\{   \Big|\frac{\partial \rho_i}{\partial\tau}\Big|^2 +\frac{d_i^2}{r^2}  (\rho_i^2-\al_i^2) + \rho_i^2 \Big| \frac{\partial \psi_i}{\partial \tau} \Big|^2 \Big\} +  \int_{A_R} \frac{2d_i\rho_i^2 }{r} \frac{\partial \psi_i}{\partial \tau}  + \int_{A_R}  \frac{\al_i^2 d_i^2}{r^2}.
\end{align*} 
The first integral is uniformly bounded with respect to $R$ by \eqref{eq:nabla psi bdd},   \eqref{eq:nabla rho bdd} and the conditions  $I_{\al_i}(u_i)<\infty$.
Moreover,
\[ \int_{A_R} \frac{\rho_i^2}{r}  \frac{\partial \psi_i}{\partial \tau} =  \int_{A_R} \frac{\rho_i^2 -\al_i^2}{r}  \frac{\partial \psi_i}{\partial \tau} \le \frac{1}{R_0} I_{\al_i}(u_i)^{\frac12} \Big( \int_{B_{R_0}^c} |\nabla \psi_i|^2 \Big)^{\frac12} \le C.
\] 
Consequently, we have
\begin{equation}
\label{eq:tang der on A_R}
 \int_{A_R} \Big|\frac{\partial u_i}{\partial\tau}\Big|^2 = 2\pi \al_i^2 d_i^2 \log R + O(1) \qas R \to \infty.
\end{equation}
Then, we obtain the desired result \eqref{eq:quant P3_n} from \eqref{eq:E final step} and \eqref{eq:tang der on A_R}.
It remains to prove    \eqref{eq:nabla psi bdd}, \eqref{eq:nabla rho bdd}, and \eqref{eq:remain}.\\

{\bf Proof of  \eqref{eq:nabla psi bdd}:}
Inserting \eqref{eq:express} in  \eqref{eq:SGL-rtwo}  and taking real and imaginary parts, we obtain that  for each $i$ and $|x| \ge R_0$,
\begin{align}
\label{eq:div}
\nabla \cdot \big(\rho_i^2\nabla\vp_i\big)&=\rho_i(\rho_i\Delta\vp_i+2\nabla\rho_i\cdot\nabla\vp_i)=0,\\
\label{eq:rho_j}
-\la_i \Delta\rho_i+\la_i \rho_i|\nabla\vp_i|^2&= \rho_i\Big(n-\sum_{j=1}^n\rho_j^2\Big).
\end{align}
Letting $x^\perp =(-x_2,x_1)$ for $x=(x_1,x_2)$, we can rewrite \eqref{eq:div} as
\begin{equation}
\label{eq:div-2}
\nabla \cdot \Big(\rho_i^2\Big(\frac{d_i}{r^2}x^\perp+\nabla\psi_i\Big)\Big)=0 \qfor |x| \ge R_0.
\end{equation}
Integrating \eqref{eq:div-2} on $\rtwo \setminus B_R$ for $R>R_0$, we are led to
\begin{equation}
\label{eq:psi-normal-int}
 0 = \int_{S_R}\rho_i^2\Big(\frac{d_i}{r^2}x^\perp+\nabla\psi_i\Big) \cdot \nu
 =  \int_{S_R}\rho_i^2\frac{\partial\psi_i}{\partial\nu}.
 \end{equation}
 Here, $\nu$ is the outward unit normal vector to $\pa B_R$.
 Let
\begin{align*}
\psi_i^R&=\frac{1}{2\pi R}\int_{S_R}\psi_i dx\qfor i=1,\cdots,l\\
\eta_0(R)&=  \sum_{i=1}^l \int_{A_R}  (\al_i^2-\delta) |\nabla\psi_i|^2 dx.
\end{align*} 
Multiplying \eqref{eq:div-2} by $(\psi_i-\psi_i^R)$ and integrating it over $A_R $ for each $i$,   we have  by \eqref{eq:behavior on inf two} and \eqref{eq:psi-normal-int}
 \begin{equation}\label{eq:f(R) ineq}
\begin{aligned}
 \eta_0(R) & \le \sum_{i=1}^l\int_{A_R} \rho_i^2|\nabla\psi_i|^2~+~ \sum_{i=1}^l\left\{ - \int_{S_{R_0}} \rho_i^2 \frac{\partial\psi_i}{\partial\nu} \psi_i~+ \right. \\
   &\quad\quad  \left. \int_{S_R} \rho_i^2\frac{\partial\psi_i}{\partial\nu}(\psi_i-\psi_i^R)
     +   \int_{A_R} \Big[  \frac{d_i}{r} (\al_i^2-\rho_i^2) \frac{\partial\psi_i}{\partial\tau}    \Big] \right\}
       \\
      &  =  (I)+(II)+(III).
\end{aligned}
\end{equation} 
By the Poincar\'{e} inequality
\[\int_{S_R}|\psi_i-\psi_i^R|^2\le R^2\int_{S_R} \Big|\frac{\partial\psi_i}{\partial\tau}\Big|^2,
\]
we obtain
\begin{align*}
\int_{S_R}   \Big|\frac{\partial\psi_i}{\partial\nu}\Big| \cdot |\psi_i-\psi_i^R|  &\le  \Big(\int_{S_R}   \Big|\frac{\partial\psi_i}{\partial\nu}\Big|^2 \Big)^{\frac12} \Big( R^2\int_{S_R} \Big|\frac{\partial\psi_i}{\partial\tau}\Big|^2 \Big)^{\frac12} \\
& \le \frac{R}{2} \int_{S_R} \Big( \Big|\frac{\partial\psi_i}{\partial\nu }\Big|^2 + \Big|\frac{\partial\psi_i}{\partial\tau}\Big|^2 \Big) = \frac{R}{2} \int_{S_R} |\nabla \psi_i|^2.
\end{align*}
Hence, by the choice of $\mu$
\begin{align*}
(II)& \le \mu \sum_{i=1}^l \int_{S_R} (\al_i^2-\delta)  \Big|\frac{\partial\psi_i}{\partial\nu}\Big| |\psi_i-\psi_i^R|  \\
&\le \frac12  \mu R  \sum_{i=1}^l  \int_{S_R}   (\al_i^2-\delta) |\nabla \psi_i|^2   =  \frac12  \mu R  \eta_0'(R).
\end{align*}
Furthermore, since $I_{\al_i}(u_i)<\infty$,   it comes from Young's inequality that
\begin{align*}
(III)&\le  \sum_{i=1}^l \frac{d_i}{R_0}  I_{\al_i}(u_i)^{\frac12}  \Big(\int_{A_R} |\nabla \psi_i|^2 \Big)^{\frac12}   \le C_0+ \delta  \eta_0(R). 
\end{align*}
As a consequence, we deduce from \eqref{eq:f(R) ineq} that
\begin{equation}
\label{eq;eta0 ineq}
 \eta_0(R) \le \frac{\mu R}{2(1-\delta)}\eta_0'(R)+  \frac{C_0}{1-\delta} \equiv \frac{R}{\nu} \eta_0'(R)+C_1,
\end{equation}
where $C_0$ and $C_1$ are independent of $R$.
We claim that $\eta_0(R)\le C_1$ for all $R>R_0$.
Otherwise, there exists $R_1>R_0$ such that $\eta_0(R_1)>C_1$.
Set $\eta_1(R)=\eta_0(R)-C_1$.
Then, we get $[R^{-\nu}\eta_1(R)]'>0$, which implies that
\[ \eta_1(R) \ge \Big(\frac{R}{R_1} \Big)^{\nu} \eta_1(R_1) \qfor R>R_1.
\]
However, we see from \eqref{eq:int}  that $\eta_1(R)\le C R$ for some $C>0$ and for all large  $R>R_1$.
Since $\nu = 2(1-\delta) /\mu>1$ by the choice of $\delta$ and $\mu$, this yields a contradiction.\\

{\bf Proof of  \eqref{eq:nabla rho bdd}:}
 Multiplying \eqref{eq:rho_j}    by $(\al_i-\rho_i)\xi_R$ and integrating it on $\rtwo\setminus   B_{R_0}$, we have
\[\begin{aligned}
& \la_i \int_{\rtwo\setminus   B_{R_0} } |\nabla\rho_i|^2\xi_R\\
\le&~ \la_i\int_{S_{R_0}}\Big|\frac{\partial\rho_i}{\partial\nu}\Big|\cdot |\al_i-\rho_i| +\frac{ \la_i }{2}\Big| \int_{\rtwo\setminus   B_{R_0} }\nabla\xi_R\cdot\nabla(\al_i-\rho_i)^2\Big|\\
&+ \la_i\int_{\rtwo\setminus   B_{R_0} }|\rho_i(\al_i-\rho_i)|\cdot|\nabla\vp_i|^2+ \int_{\rtwo\setminus   B_{R_0} }\Big| \rho_i (\al_i-\rho_i)\Big(n-\sum_{j=1}^n\rho_j^2\Big)\Big|\xi_R\\
\le &~C \Big\{ 1+\frac{1}{R^2}  I_{\al_i} (u_i) +     \| \nabla \vp_i\|_{L^2({\rtwo\setminus   B_{R_0} }) }+     I_{\al_i}(u_i)^{1\over2} E^{1\over2 } \Big\} <\infty,
\end{aligned}
\]
where the last inequality comes from $E<\infty$, $I_{\al_i}(u_i)<\infty$, and \eqref{eq:nabla psi bdd}.\\

{\bf Proof of  \eqref{eq:remain}:} \\
Let $i \in \{ l+1,\cdots,n\}$ be fixed.
We remind that 
\begin{equation}\label{eq:I-alpha=0}
I_{\al_i} = \int_\rtwo |u_i|^2 = \int_0^\infty \int_{\partial B_r }  |u_i|^2dS dr <\infty.
\end{equation}
So, there exists $r_n \to\infty$ such that 
\[\int_{\partial B_{r_n} }  |u_i|^2dS <\frac{1}{r_n}.
\]
We multiply \eqref{eq:SGL-rtwo} by $\xi_{r_n}^2u_i$ and integrate on $B_{r_n}$, and then we have
\begin{align*}
& \lambda_i\int_{B_{r_n}} \xi_{r_n}^2|\nabla u_i|^2+\lambda_i\int_{B_{r_n}} 2\xi_{r_n} u_i \nabla u_i \cdot \nabla\xi_{r_n}\\
=~ &\lambda_i\int_{\partial B_{r_n}}\frac{\partial u_i}{\partial\nu}\xi_{r_n}^2u_i+\int_{B_{r_n}}\xi_{r_n}^2u_i^2(n-f).
\end{align*}
The right side of the above is uniformly bounded by  \eqref{eq:gradient L-infty} and  \eqref{eq:I-alpha=0}.
Hence, we are led to 
\[\begin{aligned}\lambda_i\int_{\rtwo}\xi_{r_n}^2 |\nabla u_i|^2&\le C\int_{\{{r_n}<|x|<2{r_n}\}}|\xi_{r_n}\nabla u_i||\nabla \xi_{r_n}|\\
&\le C\Big(\int_{\{{r_n}<|x|<2{r_n}\}}\xi_{r_n}^2|\nabla u_i|^2\Big)^{\frac{1}{2}}\Big(\int_{\{{r_n}<|x|<2{r_n}\}}|\nabla \xi_{r_n}|^2\Big)^{\frac{1}{2}}\\
&\le C\Big(\int_{\rtwo}\xi_{r_n}^2|\nabla u_i|^2\Big)^{\frac{1}{2}}.
\end{aligned}\]
Letting $r_n\to \infty$, we obtain the  desired result \eqref{eq:remain}. 
\end{proof}



 \subsubsection*{Acknowledgements.}
 Jongmin Han was supported by Basic Science Research Program through
the National Research Foundation of Korea(NRF) funded by the Ministry of Education (2018R1D1A1B07042681).
Juhee Sohn was supported by the National Research Foundation of Korea(NRF) grant funded by the Korea government(MSIT) (2021R1G1A1003396).

\small
 \bibliographystyle{amsplain}

\end{document}